\documentclass[a4paper,oneside,11pt]{amsart}

\usepackage[english]{babel}
\usepackage{amsmath}
\usepackage{amssymb}

\usepackage{mathrsfs}
\usepackage{amssymb}
\usepackage{amsthm}
\usepackage{enumerate}
\usepackage{ifthen}
\usepackage{bbm}
\usepackage{color}
\usepackage{xcolor}
\usepackage{graphicx}
\usepackage{hyperref}
\usepackage{soul}
\usepackage{geometry}
\newcommand{\margnote}[1]{
\ifthenelse{\boolean{shownotes}}%
{\marginpar{\raggedright\tiny\texttt{#1}}}%
{}%
}

\newcommand{\hole}[1]{
\ifthenelse{\boolean{shownotes}}%
{\begin{center} \fbox{ \rule {.25cm}{0cm}
\rule[-.1cm]{0cm}{.4cm} \parbox{.85\textwidth}{\begin{center}
\texttt{#1}\end{center}} \rule {.25cm}{0cm}}\end{center}}
{}
}
\newtheorem{thm}{Theorem}[section]
\theoremstyle{plain}
\newtheorem{lemma}[thm]{Lemma}
\theoremstyle{plain}
\newtheorem*{thm*}{Theorem}
\theoremstyle{plain}
\newtheorem{prop}[thm]{Proposition}
\theoremstyle{plain}

\theoremstyle{definition}
\newtheorem{definition}[thm]{Definition}    
\theoremstyle{remark}

\newcommand{\T}{\mathbb{T}}
\newcommand{\R}{\mathbb{R}}
\newcommand{\Z}{\mathbb{Z}}
\newcommand{\N}{\mathbb{N}}
\newcommand{\ua}{u^{\alpha}}
\newcommand{\un}{u^{N}}
\newcommand{\omn}{\omega^{N}}

\newcommand{\oa}{\omega^{\alpha}}
\newcommand{\weakto}{\rightharpoonup}
\newcommand{\weaktos}{\stackrel{*}{\rightharpoonup}}
\newcommand{\Cs}{C_{u_0}^{s,T}}
\newcommand{\Csn}{C_{u^N_0}^{s,T}}
\newcommand{\curl}{\mathop{\mathrm {curl}}}
\newcommand{\diver}{\mathop{\mathrm {div}}}
\usepackage{changes}

\DeclareMathOperator{\lapl}{\Delta}
\DeclareMathOperator{\idnt}{I}
\numberwithin{equation}{section}


\begin{document}

\title[Convergence of the Euler-Voigt approximation]{Convergence of the Euler-Voigt equations to the Euler equations in two dimensions}

\author[S. Abbate]{Stefano Abbate}
\address[S. Abbate]{Gran Sasso Science Institute (GSSI), Viale Francesco Crispi 7, I-67100, L'Aquila, Italy}
\email[]{\href{stefano.abbate@}{stefano.abbate@gssi.it}}

\author[L. C. Berselli]{Luigi C. Berselli}
\address[L. C. Berselli]{Dipartimento di Matematica, Universit\`a di Pisa, Via F. Buonarroti 1/c, I56127, Pisa, Italy}
\email[]{\href{luigi.carlo.berselli@}{luigi.carlo.berselli@unipi.it}}

\author[G. Crippa]{Gianluca Crippa}
\address[G. Crippa]{Departement Mathematik und Informatik, Universität Basel, Spiegelgasse 1, CH-4051, Basel, Switzerland}
\email[]{\href{gianluca.crippa@}{gianluca.crippa@unibas.ch}}

\author[S. Spirito]{Stefano Spirito}
\address[S. Spirito]{DISIM - Dipartimento di Ingegneria e Scienze dell'Informazione e Matematica\\ Universit\`a  degli Studi dell'Aquila \\Via Vetoio \\ 67100 L'Aquila \\ Italy}
\email[]{\href{stefano.spirito@}{stefano.spirito@univaq.it}}

\begin{abstract}
  In this paper, we consider the two-dimensional torus and we study
  the convergence of solutions of the Euler-Voigt equations to
  solutions of the Euler equations, under several regularity
  settings. More precisely, we first prove that for weak solutions of the
  Euler equations with vorticity in~$C([0,T];L^2(\T^2))$ the
  approximating velocity converges strongly in
  $C([0,T];H^1(\T^2))$. Moreover, for the unique Yudovich solution of
  the $2D$ Euler equations we provide a rate of convergence for the
  velocity in~$C([0,T];L^2(\T^2))$. Finally, for classical solutions
  in higher-order Sobolev spaces we prove the convergence with
  explicit rates of both the approximating velocity and the
  approximating vorticity in $C([0,T];L^2(\T^2))$.
\end{abstract}

\maketitle

\section{Introduction}
Let $\mathbb{T}^2$ be the two-dimensional flat torus and let
$T>0$. The $2D$ Euler equations for an incompressible
inviscid fluid are
\begin{flalign}
	\begin{cases}
		\partial_t u +u\cdot \nabla u = -\nabla p &\quad \text{on}  \quad(0,T) \times \mathbb{T}^2 \\
		\diver u = 0&\quad \text{on}\quad (0,T) \times \mathbb{T}^2 \\
		u(0, \cdot) = u_0&\quad \text{on} \quad\mathbb{T}^2,
	\end{cases}\label{eq:e}
\end{flalign}
where the unknowns are the velocity field $u:(0,T)\times\T^{2}\mapsto\R^{2}$ and the
scalar pressure $p:(0,T)\times\T^{2}\mapsto\R$. The initial datum
$u_{0}:\T^{2}\mapsto\R^{2}$ is a given divergence-free vector field with zero average.

In this paper, we are interested in studying convergence properties of
a large scale approximation (of the $\alpha$-family) of the system
\eqref{eq:e}, introduced in the literature to obtain reliable
simulations of turbulent flows. Indeed, as a turbulent regime is
approached, producing a {\em Direct Numerical Simulation (DNS)} of the
flow becomes computationally prohibitive, especially for
three-dimensional fluids. A natural alternative is that of
\textit{averaging} the equations in some way and then to
consider/simulate the flow only for large scales,
see~\cite{GOP04}. Models obtained by the procedure described above are
often called {\em Large Eddy Simulations (LES)} models and they have
been extensively studied in literature, see for example \cite{GOP04,
  LL03, LL06, BB12}. In particular, we refer to \cite{BIL06} for a
more complete overview of LES models and some theoretical results.

In the context of the incompressible Euler equations two main models have been
well-studied: the $\alpha$-Euler equations introduced in \cite{H02, HMR98}, and the
Euler-Voigt equations introduced in \cite{CLT06}. The Euler-Voigt model is the main focus
of the present paper and is given by the following system
\begin{equation}\label{eq:ev}
\begin{cases}
\partial_t u^\alpha-\alpha\Delta\partial_t\ua +u^\alpha\cdot \nabla u^\alpha= -\nabla p^\alpha&\quad \text{on}  \quad(0,T) \times \mathbb{T}^2 \\
\diver u^\alpha  = 0&\quad \text{on}\quad (0,T) \times \mathbb{T}^2 \\
u^\alpha(0, \cdot) = u^\alpha_0 &\quad \text{on} \quad\mathbb{T}^2. 
\end{cases}
\end{equation}
In \eqref{eq:ev} the unknowns are
$u^\alpha:(0,T) \times \mathbb{T}^2 \mapsto \mathbb{R}^2$ and
$p^\alpha: (0,T) \times \mathbb{T}^2 \mapsto \mathbb{R}$, and
$u^\alpha_0:\T^{2}\mapsto\R^{2}$ is a zero average divergence-free
initial datum. The parameter $\alpha>0$ has the dimensions of a
squared length and is connected with the square of the smallest
resolved scale. We write simply Euler-Voigt and not
``$\alpha$-Euler-Voigt'' to avoid any possible confusion with the
$\alpha$-Euler model, even if \eqref{eq:ev} is a family of
problems parameterized by the parameter $\alpha>0$.

To briefly explain the derivation of the Euler-Voigt equation, we
begin by applying the differential filtering operator obtained by
inverting (in the periodic setting) the Helmholtz operator
\begin{equation*}
	\overline{f}:= (\idnt-\alpha \lapl)^{-1} f,\qquad \text{for some} \enskip
        \alpha>0,
\end{equation*} 
and the ``filter'' action --denoted by $\overline{(\cdot)}$-- is linked with the damping of the Fourier
coefficients since  they satisfy the following equality
\begin{equation*}
  \widehat{\overline{f}}_{k}
  =\frac{
          \widehat{f}_k}{1+\alpha|k|^2}\qquad \text{for all 
          wave-numbers }k\in\Z^{3}. 
\end{equation*}
Applying then the Helmholtz filtering  to each scalar
component of the (nonlinear)
Euler equations \eqref{eq:e}, yields
\begin{flalign*}
	\begin{cases}
		\partial_t \overline{u} + \nabla\cdot( \overline{u \otimes u}) = -\nabla \overline{p}, &\quad \text{on}  \quad(0,T) \times \mathbb{T}^2 \\
		\diver \overline{u}= 0, &\quad \text{on}  \quad(0,T) \times \mathbb{T}^2.
	\end{cases}
\end{flalign*}
However, this system is not closed, since it does not depend only on
the filtered variables $(\overline{u},\overline{p})$, due to the
presence of a nonlinear convective term. Thus, the standard
approach is to introduce the Reynolds stress tensor 
$\mathcal{R}(u, u) := \overline{u \otimes u}- \overline{u} \otimes
\overline{u}$. Specifically, we add and subtract
$\nabla\cdot (\overline{u} \otimes \overline{u})$, yielding
\begin{flalign*}
	\begin{cases}
		\partial_t \overline{u} + \nabla\cdot( \overline{u} \otimes\overline{u})+\nabla\cdot  \mathcal{R}(u, u) = -\nabla \overline{p}, &\quad \text{on}  \quad(0,T) \times \mathbb{T}^2 \\
		\diver \overline{u}= 0, &\quad \text{on}  \quad(0,T) \times \mathbb{T}^2.
	\end{cases}
\end{flalign*}
Then, if one considers the simplified Bardina-like approximation of
the Reynolds stress tensor
$\mathcal{R}(u, u)\sim \overline{\overline{u} \otimes \overline{u}}-
\overline{u} \otimes \overline{u}$, introduced and studied in
\cite{BFR80}, the Euler-Voigt equations \eqref{eq:ev} for a give
$\alpha>0$ is obtained, see also \cite{CLT06}. An more general
approach within the framework of approximate deconvolution models is
also presented in \cite{BKR2016}.

One of the most important properties of the Euler-Voigt equations (but the same
is valid also for its viscous counterpart, the Navier-Stokes-Voigt
equations) is that the filtering \& modeling is not introducing
extra dissipation, but a dispersive effect, which allows for global in time existence
and uniqueness of weak solutions, also in
the three dimensional case.\\
Our main objective is to rigorously study the ``consistency of the
model'', \textit{i.e.}, the convergence as $\alpha\to 0^{+}$ of
solutions of \eqref{eq:ev} towards solutions of the $2D$ 
Euler equations \eqref{eq:e} in different regularity settings and
also, when possible, to provide rates of convergence.

A main difficulty in studying the convergence of solutions of
\eqref{eq:ev} to the corresponding solutions of \eqref{eq:e} lies in
the structure of the approximating vorticity equation. A very special
property of the $2D$ Euler equations, which has relevant consequences
on the proofs of several results, is that the vorticity
$\omega:=\curl\,u=\partial_{1}u_{2}-\partial_{2} u_{1}$ formally
satisfies the following scalar equation
\begin{equation}\label{eq:transpvort}
\partial_t\omega+u\cdot\nabla\omega=0.
\end{equation}
The equation \eqref{eq:transpvort} is a nonlocal transport equation
with a divergence-free vector field $u$ advective the vorcitity $\omega$. This implies that all the
$L^p$-norms of $\omega$, with $p\in[1,\infty]$, are formally conserved. On the other hand, for the Euler-Voigt equations, defining $\oa:=\curl\ua$ it holds that 
\begin{equation}\label{eq:transpvortalpha}
\partial_t\oa-\alpha\Delta\partial_t\oa+\ua\cdot\nabla\oa=0.
\end{equation}

The equation \eqref{eq:transpvortalpha} is not anymore a transport
equation for $\omega^{\alpha}$. In particular, while it holds that
\begin{equation*}
\|\oa(t)\|_2^2+\alpha\|\nabla\oa(t)\|_2^2=\|\oa_0\|_2^2+\alpha\|\nabla\oa_0\|_2^2,
\end{equation*}
it is not clear how to get the analogous bound 
$\oa \in L^{\infty}(0,T;L^{p}(\T^2))$, uniformly in $\alpha$, when
$p\not=2$. As a consequence, some convergence results, specifically in
weak regularity settings, are different from the ones obtained by
other approximation methods which are based on the use of the
transport equation 
for the approximate vorticity as, \textit{e.g.}, for the
$\alpha$-Euler approximation, see \cite{ACS24, BI17, BILL20, LT10,
  LLTZ15}.

We now state our main results and we compare them with the ones
available for other approximation schemes, notably the $\alpha$-Euler
approximation. For the notations and the definitions appearing in the
following theorems we refer to Section \ref{sec:prelim}. Concerning
the initial datum of \eqref{eq:ev} we assume that for any $\alpha>0$,
$u_0^{\alpha}$ are zero-average divergence-free vector fields such
that
\begin{align}
&\bullet\,\,u_0^{\alpha}\in C^{\infty}(\T^2)\mbox{ for any fixed }\alpha>0,\label{eq:aid1}\\
&\bullet\,\,\{u_0^{\alpha}\}_{\alpha}\subset H^{1}(\T^2),\quad\{\sqrt{\alpha}\nabla\omega_0^{\alpha}\}_{\alpha}\subset L^{2}(\T^2), \mbox{ uniformly in }\alpha>0.\label{eq:aid2}
\end{align}
 The first result of the paper concerns the convergence of solutions
 of \eqref{eq:ev} towards solutions of \eqref{eq:e} in the setting of
 weak solutions with finite enstrophy, namely weak solutions of
 \eqref{eq:e} such that  $\omega\in L^{\infty}(0,T;L^{2}(\T^2))$. 
\begin{thm}\label{teo:weak}
Let $u_0$ be a zero average divergence-free vector field such that~$\omega_0=\curl u_0\in L^{2}(\T^2)$. Let $\{\ua_0\}_{\alpha}$ satisfy \eqref{eq:aid1} and \eqref{eq:aid2} and assume in addition that 
\begin{equation}\label{eq:convin}
\begin{aligned}
&\ua_0\to u_0\mbox{ in }H^{1}(\T^2),\quad&\sqrt{\alpha}\nabla\omega^{\alpha}_0\to 0\mbox{ in }L^{2}(\T^2). 
\end{aligned}
\end{equation}
Then, up to subsequences, there exists $u\in C([0,T];H^{1}(\T^2))$ such that 
\begin{align}
&\ua\to u\mbox{ in }C([0,T];H^1(\T^2)),\label{eq:convuth}\\
&\omega^{\alpha}\to \omega\mbox{ in }C([0,T];L^2(\T^2)).\label{eq:convoth}
\end{align}
Moreover, $u$ is a weak solution of \eqref{eq:e} and 
\begin{align}
&\|u(t)\|_2=\|u_0\|_2,\enskip \mbox{ for any }t\in[0,T],\label{eq;energyth}\\
&\|\omega(t)\|_2=\|\omega_0\|_2, \enskip \mbox{ for any }t\in[0,T].\label{eq;enstrophyth}
\end{align}
\end{thm}
We note that in general for any $p\in(1,\infty]$ it is possible to construct weak solutions of
\eqref{eq:e} with vorticity in $L^{\infty}(0,T;L^{p}(\T^2))$, see \cite{DPM87}. Moreover, for solutions
which are limit of various approximation schemes, including
the $\alpha$-Euler approximation, see \cite{C22, CLLS16, CCS20,
  ACS24}, it is possible to prove the conservation of the kinetic energy. For the end-point case $p=1$
 and the case of positive Radon measure
it is possible to prove the existence of weak
solutions, see \cite{VW93} and \cite{D91} respectively, and the lack of anomalous dissipation in the vanishing
viscosity limit, \cite{DRP24}. On the other hand, in the case of the Euler-Voigt
equations, in Theorem \ref{teo:weak} we are able to consider only the
case $p=2$.\\

Next, we remark in the class of weak solutions of \eqref{eq:e} with vorticity in $L^{\infty}(0,T;L^{p}(\T^2))$ uniqueness fails for some $p$ close to one, see
\cite{BCK24}, and for any $1\leq p<\infty$ if one considers external forces
in \eqref{eq:e}, see \cite{V18, ABCGJK24}. On the other hand, uniqueness is known for $p=\infty$, namely in the class of Yudovich weak solutions, see \cite{Y63}.  Therefore, in this
class it is reasonable to expect that in addition some rate of
convergence may be proved. This is indeed the content of the following
result.
\begin{thm}\label{teo:yudovich}
Let $u_0$ be a zero-average divergence-free vector field such that
$\omega_0=\curl u_0\in L^{\infty}(\T^2)$ and let $u$ be the unique Yudovich solution of \eqref{eq:e}. Let $\{\ua_0\}_{\alpha}$ satisfy \eqref{eq:aid1} and \eqref{eq:aid2}. Then, given $T>0$ there exist  constants $C_1$, $C_2$, and $C_3$, depending only on $T$ and $\|\omega_0\|_{\infty}$, such that 
\begin{equation}\label{eq:stabilityth}
\sup_{0 \leq t \leq T}\|\ua-u\|_2\leq C_3
(\sqrt{\alpha}+\|\ua_0-u_0\|_2^{2})^{\displaystyle{\frac{\mathrm{e}^{-C_1T}}{2}}},
\end{equation}
provided that $\sqrt{\alpha}+\|\ua_0-u_0\|_2^{2}\leq C_2$. 
\end{thm}
The rate of convergence \eqref{eq:stabilityth} is the analogous to the
one obtained by Chemin in \cite{C96} for the vanishing viscosity
limit. Similar results are also true for other approximation schemes,
such as the $\alpha$-Euler \cite{ACS24} and the Fourier-Galerkin
method \cite{BS24}. 

On the other hand, the situation differs, particularly when compared
to the case of $\alpha$-Euler, if one is interested in obtaining
rates for
the convergence of the vorticity \eqref{eq:convoth}. Indeed, for
$\alpha$-Euler, it is possible to give an explicit rate for the
convergence \eqref{eq:convoth} provided that $\omega_0\in
L^{\infty}(\T^{2})\cap B^{s}_{p,\infty}(\T^2)$, with $s>0$ and $p\geq
1$, see \cite{ACS24}. The result in \cite{ACS24} is the analogous of
the results obtained for the vanishing-viscosity approximation in
\cite{CDE22,CCS21}. The crucial point is that again some transport structure
is preserved in the approximating vorticity in the $\alpha$-Euler.

On the contrary --at present-- for the Euler-Voigt equations
\eqref{eq:ev}, even if we consider
$\omega_0\in L^{\infty}(\T^{2})\cap B^{s}_{p,\infty}(\T^2)$, we are
not able to quantify the strong convergence \eqref{eq:convoth}. In the
next theorem, we show that a rate of convergence for
\eqref{eq:convoth} can be proved, if we consider more regular initial
data.
\begin{thm}\label{teo:strong}
Let $s>2$ and $u_0\in H^{s}(\T^2)$ be a zero-average divergence-free vector field and let $T>0$, arbitrary and finite. Let $u$ be the unique solution of \eqref{eq:e} in $C([0,T];H^{s}(\T^{2}))$. Let $\{\ua_0\}_{\alpha}$ satisfy \eqref{eq:aid1} and \eqref{eq:aid2}. Then, 
\begin{equation}\label{eq:convuths}
\sup_{0 \leq t \leq T}\|\ua-u\|_2\lesssim (\alpha+\|\ua_0-u_0\|_2^{2})^{\frac{1}{2}}.
\end{equation}
Moreover, if $s\geq 3$ 
\begin{equation}\label{eq:convuths1}
\sup_{0 \leq t \leq T}\|\omega^{\alpha}-\omega\|_2\lesssim (\alpha+\|\omega^{\alpha}_0-\omega_0\|_2^{2}+\alpha\|\nabla\omega^{\alpha}_0-\nabla\omega_0\|_2^2)^{\frac{1}{2}},
\end{equation}
while if $2<s<3$ 
\begin{equation}\label{eq:convuths2}
\sup_{0 \leq t \leq T}\|\omega^{\alpha}-\omega\|_2\lesssim (\alpha^{\frac{s-1}{2}}+\|\omega^{\alpha}_0-\omega_0\|_2^{2}+\alpha\|\nabla\omega^{\alpha}_0-\nabla\omega_0\|_2^2)^{\frac{1}{2}}.
\end{equation}
\end{thm}
We note that the case $s\geq 3$ in Theorem \ref{teo:strong} is the
two-dimensional analog of the convergence result in \cite[Theorem
5.1]{LaT10} and we recover the same rate of convergence. In addition,
extending \cite[Remark 2]{LaT10}, in Theorem \ref{teo:strong} we also
consider the case $2<s<3$, namely the full range of $s\in\R$ such that
for $\gamma\in(0,1)$the embedding of $H^{s}(\T^2)\subset C^{1,\gamma}(\T^{2})$ holds. In this case, an additional regularization
and a careful estimate on the growth of the higher-order norms of
the solution of \eqref{eq:e} are needed, and this explains the
difference between the rates in \eqref{eq:convuths1} and in
\eqref{eq:convuths2}.  
\subsection*{Organization of the paper} In Section \ref{sec:prelim},
we introduce the notations used throughout the paper, we recall some
classical results on the Euler equations, and we state the necessary
ones on the Euler-Voigt equations. Then in the subsequent Sections
\ref{sec:3},~\ref{sec:4}, and \ref{sec:5} we give the proofs of Theorems
\ref{teo:weak}, \ref{teo:yudovich}, and \ref{teo:strong}, respectively. 
\section{Preliminaries}\label{sec:prelim}
\subsection{Notations}
Throughout the paper, we consider as a domain the two-dimensional flat torus, which is defined as $\mathbb{T}^2 := \mathbb{R}^2/2\pi\mathbb{Z}^2$. In particular, if the domain of an integral is not explicitly stated, it is assumed to be the torus. The space $\mathcal{C}^\infty_c([0,T) \times \mathbb{T}^2)$ denotes the space of smooth functions with compact support in time and periodic in space. The standard Lebesgue and Sobolev spaces are denoted by $L^p(\mathbb{T}^2)$, $H^s(\mathbb{T}^2)$ with $s\in\R$, and $W^{k,p}(\mathbb{T}^2)$ with $k\in\N$ and $1\leq p\leq \infty$. We denote their norms with $\|\cdot\|_p$, $\|\cdot\|_{s,2}$, and $\|\cdot\|_{k,p}$, respectively. Moreover, we denote the scalar product in the Hilbert space $L^2(\mathbb{T}^2)$ as $(\cdot, \cdot)$. Given a Banach space $X$, the classical Bochner spaces are denoted by $L^p(0,T;X)$. Moreover, $C([0,T]; X_{w})$ denotes the space of continuous function with values in $X$ endowed with the weak topology. Finally, we adopt the notation $\lesssim$, meaning that the terms on the right-hand side are bounded up to a constant factor that is independent of $\alpha$. \\
\subsection{On the $2D$ Euler equations}
In this section we recall some of the results concerning the $2D$ Euler equations which will be used in the sequel. We start with the following one concerning the global well-posedness in Sobolev spaces. 
\begin{thm}\label{teo:grhse}
Let $s>2$ and $u_0\in H^{s}(\T^2)$ be a divergence-free vector field with zero average. Then, there exists a unique solution of \eqref{eq:e}
\begin{equation*}
u\in C([0,T];H^{s}(\T^2))\cap C^{1}([0,T];H^{s-1}(\T^2)).
\end{equation*}
In addition, 
\begin{equation}\label{eq:dexp}
\sup_{0 \leq t \leq T}\|u(t)\|_{s,2}\leq C\|u_{0}\|_{s,2}\exp\left((1+2\log^{+}\|u_0\|_{s,2})\,\mathrm{e}^{\,Ct\|\omega_0\|_{\infty}}-1\right),
\end{equation}
with $C>0$ depending only on $s$.
\end{thm}
We remark that since we are on the torus the pressure can be recovered by solving the associated elliptic problem. 
The proof of Theorem \ref{teo:grhse} can be found in \cite{K86}, while the inequality \eqref{eq:dexp} has been proved in \cite{D15}. 
We also consider solutions in weaker regularity settings. 
\begin{definition}
	\label{def:weaksol}
A divergence-free vector field $u \in C([0,T];L_{w}^2(\mathbb{T}^2))$ is a weak solution of \eqref{eq:e}, if  for any $\phi \in H^s({\mathbb{T}}^2)$ with $\diver\phi = 0$ and $s>2$, it holds
 \begin{equation*}
    (u(t),\phi)-(u_0,\phi)= \int_{0}^t (u(\tau) \cdot \nabla\phi, u(\tau))\,d\tau, \quad \mbox{ for any }t \in [0,T).
 \end{equation*}
 If in addition the vorticity $\omega\in L^{\infty}([0,T]\times\T^2)$, then $u$ is called a Yudovich weak solution. 
\end{definition}
In the next proposition, we summarize the results we need in the sequel concerning Yudovich solutions. 
\begin{prop}\label{prop:yud}
Let $u_0\in L^{2}(\T^2)$ be a zero-average vector field such that $\omega_0\in L^{\infty}(\T^{2})$. Then, there exists a unique Yudovich weak solution of \eqref{eq:e}. Moreover, 
\begin{itemize}
\item $u\in C([0,T];L^{2}(\T^2))$ and for any $t\in[0,T]$ 
\begin{equation}\label{eq:energy}
\|u(t)\|_{2}=\|u_0\|_2. 
\end{equation}
\item $\omega\in C([0,T];L^{2}(\T^{2}))$ and for any $t\in[0,T]$ 
\begin{equation}\label{eq:enstrophy}
\|\omega(t)\|_{2}=\|\omega_0\|_{2}.
\end{equation}
\item $u\in \mathcal{C}(0,T;W^{1,p}(\T^{2}))$ for any $1<p<\infty$, and for any $p>2$ 
\begin{equation}\label{eq:cz}
\|\nabla\,u(t)\|_{p}\lesssim p\,\|\omega(t)\|_{\infty}\quad\forall\, t\in[0,T]. 
\end{equation}
\end{itemize}
\end{prop}
The results stated in Proposition \ref{prop:yud} are due to Yudovich \cite{Y63} and the proof can also and be found in the classical references \cite{MB02, MP94}. 
\subsection{On the $2D$ Euler-Voigt equations}
In this section we recall the global well-posedness of the system \eqref{eq:ev} and we prove the uniform-in-$\alpha$ estimates we need. We start with the following global well-posedness result in higher-order Sobolev spaces which holds for any fixed $\alpha>0$. 
\begin{prop}\label{prop:regv}
Let $\alpha>0$ and $u^{\alpha}_0\in C^{\infty}(\T^{2})$ be a zero average divergence-free vector field. Then, there exists a unique solution $u^{\alpha}$ of \eqref{eq:ev} satisfying 
\begin{equation*}
u^{\alpha}\in C([0,T];H^{s}(\T^2))\cap C^{1}([0,T];H^{s-1}(\T^2)),\mbox{ for any }s>2. 
\end{equation*}
\end{prop}
The proof is analogous to \cite[Theorem 7.2]{CLT06} and it is based on a standard continuity argument and higher-order energy estimates. Since in \cite{CLT06} only the three-dimensional case is proved we give a brief sketch. 
\begin{proof}
We only give the {\em a priori} estimates. The energy estimate is the following 
\begin{equation*}
\frac{d}{dt}\left(\|\ua(t)\|_2^2+\alpha\|\nabla\ua\|_2^2\right)=0. 
\end{equation*}
Moreover, for any $s>2$ by a classical $H^s$-estimate we have that 
\begin{equation*}
\frac{d}{dt}\left(\|u^{\alpha}(t)\|^2_{s,2}+\alpha\|u^{\alpha}(t)\|^2_{s+1,2}\right)\lesssim \|\nabla u^{\alpha}(t)\|_{\infty}\|u^{\alpha}(t)\|^2_{s,2}. 
\end{equation*}
Then, if we can produce an {\em a priori} bound on $\nabla u^{\alpha}$ in $L^{1}(0,T;L^{\infty}(\T^{2}))$ we can conclude by using Gr\"onwall lemma. To this end we note that 
\begin{equation}\label{eq:deltaua}
\begin{aligned}
&\partial_{t}\Delta\ua-\alpha \partial_t\Delta(\Delta\ua)+\nabla\Delta p^{\alpha}+\Delta\ua\cdot\nabla\ua+2\nabla\ua\cdot\nabla\nabla\ua+\ua\cdot\nabla\Delta\ua=0.
\end{aligned}
\end{equation}
Then, by multiplying \eqref{eq:deltaua} by $\Delta\ua$, integrating by parts, and using the divergence-free condition we get 
\begin{equation*}
\begin{aligned}
\frac{d}{dt}\left(\|\Delta\ua\|_{2}^2+\alpha\|\nabla\Delta\ua\|_{2}^2\right)=&-4\int\nabla\ua\cdot\nabla\nabla\ua\Delta\ua\,dx-2\int\Delta\ua\cdot\nabla\ua\Delta\ua\,dx\\
&\lesssim \int|D^{2}\ua|^2|\nabla\ua|\,dx.
\end{aligned}
\end{equation*}
Recalling that in two dimensions it holds that, for any $f\in H^{1}(\T^2)$ with zero average,
\begin{equation}\label{eq:lad}
\|f\|_{4}\leq C \|f\|_2^{\frac{1}{2}}\|\nabla\,f\|_{2}^{\frac{1}{2}},
\end{equation}
we have  
\begin{equation*}
\begin{aligned}
\frac{d}{dt}(\|\Delta\ua\|_{2}^2+\alpha\|\nabla\Delta\ua\|_{2}^{2})&\lesssim\|\nabla\ua\|_2\|D^{2}\ua\|_{4}^{2}\\
&\lesssim\frac{1}{\alpha}\|\nabla\ua\|_2(\|D^{2}\ua\|_{2}^{2}+\alpha\|\nabla\,D^{2}\ua\|_{2}^{2}). 
\end{aligned}
\end{equation*}
Thus, by some further elementary manipulations and by using Gr\"onwall lemma we have that for some $C_{\alpha}>0$ blowing-up as $\alpha\to 0$, the following estimate holds 
\begin{equation*}
\sup_{0 \leq t \leq T}\|\nabla\,D^{2}\ua(t)\|_{2}\leq C_{\alpha}. 
\end{equation*}
Finally, by using Sobolev embeddings we have that 
\begin{equation*}
\sup_{0 \leq t \leq T}\|\nabla\ua(t)\|_{\infty}\leq C_{\alpha},
\end{equation*}
and we can conclude. 
\end{proof}
In the following proposition we prove some estimates uniform in $\alpha \in (0, 1]$. 
\begin{prop}\label{prop:unif}
Let $\{\ua_{0}\}_{\alpha}$ be a zero-average, divergence-free initial datum satisfying \eqref{eq:aid1} and \eqref{eq:aid2} and let $T>0$ arbitrary and finite. Let $T>0$ arbitrary and finite, then, there exists a constant $C>0$ independent on $\alpha$ such that 
\begin{equation}\label{eq:unifa}
\begin{aligned}
&\sup_{0 \leq t \leq T}\|\ua\|_2\leq C,\qquad &\sup_{0 \leq t \leq T}\|\omega^{\alpha}\|_{2}\leq C,\qquad \sup_{0 \leq t \leq T}\sqrt{\alpha}\|\nabla\omega^{\alpha}\|_{2}\leq C. 
\end{aligned}
\end{equation}
\end{prop}
\begin{proof}
Multiplying by $\ua$ the first equations of \eqref{eq:ev}, we have that 
\begin{equation*}
\|\ua(t)\|_{2}^{2}+\alpha\|\nabla\ua(t)\|_{2}^{2}=\|\ua_{0}\|_{2}^{2}+\alpha\|\nabla\ua_0\|_{2}^{2}. 
\end{equation*}
Thus, by using \eqref{eq:aid2} we have the first bound in \eqref{eq:unifa}. Concerning the last two bounds, by using the equation for the vorticity 
\begin{equation*}
\partial_{t}\omega^{\alpha}-\alpha\Delta\partial_t\omega^{\alpha}+\ua\cdot\nabla\omega^{\alpha}=0,
\end{equation*}
we obtain that 
\begin{equation*}
\|\omega^{\alpha}(t)\|_{2}^2+\alpha\|\nabla\omega^{\alpha}(t)\|_{2}^{2}=\|\omega^{\alpha}_0\|_{2}^{2}+\alpha\|\nabla\omega^{\alpha}_0\|_{2}^{2}, 
\end{equation*}
and thus by using \eqref{eq:aid2} we conclude. 
\end{proof}
\section{Proof of Theorem \ref{teo:weak}}\label{sec:3}
The proof is based on a compactness argument and some tools from DiPerna-Lions theory \cite{DPL89}. 
\begin{proof}[Proof of Theorem \ref{teo:weak}]
Let $\ua$ be the solution of \eqref{eq:ev} obtained in Proposition \ref{prop:regv}. By Proposition \ref{prop:unif} we have that 
\begin{equation*}
\{\ua\}_{\alpha}\subset L^{\infty}(0,T;H^{1}(\T^{2})). 
\end{equation*}
We can also easily deduce that for any $s>2$
\begin{equation*}
\{\partial_{t}\ua\}_{\alpha}\subset L^{\infty}(0,T;H^{-s}(\T^{2})),
\end{equation*}
and thus by Aubin-Lions Lemma we get that there exists 
\begin{equation*}
u\in C([0,T);L^{2}(\T^{2}))\cap L^{\infty}(0,T;H^{1}(\T^{2})),
\end{equation*}
such that 
\begin{equation}\label{eq:convebasic}
\begin{aligned}
&\ua\rightarrow u\mbox{ in }C([0,T];L^{2}(\T^{2})),\\
&\oa	\weaktos \omega\mbox{ in }L^{\infty}(0,T;L^{2}(\T^2)). 
\end{aligned}
\end{equation}
Moreover, for any $\phi\in H^{s}(\T^{2})$, with $s>2$, it holds that 
\begin{equation*}
\begin{aligned}
\sup_{0\leq t\leq T}&\alpha|(\nabla\ua(t),\nabla\phi)|\rightarrow 0,\mbox{ as }\alpha\to 0,\\
&\alpha(\nabla\ua_0,\nabla\phi)\rightarrow 0,\mbox{ as }\alpha\to 0.
\end{aligned}
\end{equation*}
Then, we can conclude that $u$ is a weak solution of \eqref{eq:e} and in addition 
\begin{equation}\label{eq:ens}
\omega\in L^{\infty}(0,T;L^{2}(\T^{2})). 
\end{equation}
To prove \eqref{eq;energyth} it is enough to recall that by \cite[Theorem 1]{CLLS16} every weak solution of \eqref{eq:e} such that $\omega\in L^{\infty}(0,T;L^{p}(\T^{2}))$ with $p>3/2$ conserves the energy. Thus, by \eqref{eq:ens} we obtain \eqref{eq;energyth}. 
Next, we prove the conservation of the enstrophy. First, given $\phi\in H^{s}(\T^2)$ with $s>2$, by Proposition \ref{prop:unif} and \eqref{eq:convebasic} we have that 
\begin{equation}\label{eq:weakvora}
(\omega^{\alpha}(t),\phi)+\alpha(\nabla\oa(t),\nabla\phi)-(\oa_0,\phi)
-\alpha(\nabla\oa_0,\nabla\phi)+\int_0^t(\oa(\tau)\ua(\tau),\nabla\phi)\,d\tau=0,
\end{equation}
converges, as $\alpha\to 0$, to 
\begin{equation}\label{eq:weakvor}
(\omega(t),\phi)-(\omega_0,\phi)+\int_0^t(\omega(\tau)u(\tau),\nabla\phi)\,d\tau=0.
\end{equation}
By \cite[Proposition 1]{LML06} it follows that $\omega\in C([0,T);L^{2}(\T^2))$ and it is a renormalized solution of the vorticity equation in the sense of DiPerna-Lions \cite{DPL89}, and thus \eqref{eq;enstrophyth} holds. Finally, we prove \eqref{eq:convoth}. Let $t\in[0,T]$, by \eqref{eq:weakvora} and \eqref{eq:weakvor} we have 
\begin{equation*}
\begin{aligned}
(\omega^{\alpha}(t)-\omega(t),\phi)&=(\oa_0-\omega_0,\phi)+\alpha(\nabla\oa_0,\nabla\phi)-\alpha(\nabla\oa(t),\nabla\phi)\\
&+\int_0^t(\omega(\tau)u(\tau)-\oa(\tau)\ua(\tau),\nabla\phi)\,d\tau.
\end{aligned}
\end{equation*}
Thus by using \eqref{eq:convebasic} and Proposition \ref{prop:unif} we get 
\begin{equation}\label{eq:pointwise}
\oa(t)\weakto\omega(t)\mbox{ in }L^{2}(\T^{2}). 
\end{equation}
In addition, by \eqref{eq;enstrophyth} and \eqref{eq:pointwise} we have that for any $t\in[0,T]$
\begin{equation*}
\begin{aligned}
\|\omega(t)\|_2&\leq \liminf_{\alpha\to 0}\|\oa(t)\|_2\\
&\leq \lim_{\alpha\to 0}\left(\|\oa_0\|_2^2+\alpha\|\nabla\oa_0\|_2^2\right)^{\frac{1}{2}}\\
&=\|\omega_0\|_2=\|\omega(t)\|_2,
\end{aligned}
\end{equation*}
thus, having weak convergence and convergence on the norms, we deduce,
\begin{equation}\label{eq:stronpoint}
\oa(t)\to\omega(t)\mbox{ in }L^{2}(\T^{2}), \enskip \forall\,t \in [0,T]. 
\end{equation} 
To upgrade the convergence \eqref{eq:stronpoint} from pointwise convergence to uniform convergence in time it is enough to prove that for any $t\in[0,T]$ and $\{t_{\alpha}\}_{\alpha}\subset[0,T]$ such that $t_{\alpha}\to t$ it holds that 
\begin{equation}\label{eq:mainta}
\oa(t_{\alpha})\to\omega(t)\mbox{ in }L^{2}(\T^{2}),\mbox{ as }\alpha\to 0. 
\end{equation} 
To prove \eqref{eq:mainta} we first note that for any $\phi\in H^{s}(\T^{2})$ with $s>2$ it holds that
\begin{equation*}
\begin{aligned}
(\omega^{\alpha}(t_{\alpha})-\omega(t),\phi)&=(\oa_0-\omega_0,\phi)+\alpha(\nabla\oa_0,\nabla\phi)-\alpha(\nabla\oa(t_{\alpha}),\nabla\phi)\\
&+\int_t^{t_{\alpha}}(\omega(\tau)u(\tau)-\oa(\tau)\ua(\tau),\nabla\phi)\,d\tau.
\end{aligned}
\end{equation*}
Then
\begin{equation*}
\begin{aligned}
|(\oa(t_{\alpha})-\omega(t),\phi)|&\lesssim \|\oa_0-\omega_0\|_2\|\phi\|_2+\alpha\|\nabla\oa_0\|_2\|\nabla\phi\|_2+\alpha\|\nabla\oa(t_\alpha)\|_2\|\nabla\phi\|_2\\
&+\int_{t}^{t_{\alpha}}\left(\|\omega(\tau)\|_2\|u(\tau)\|_2+\|\oa(\tau)\|_2\|\ua(\tau)\|_2\right)\|\nabla\phi\|_{\infty}\,d\tau.
\end{aligned}
\end{equation*}
Then, by using  Proposition \ref{prop:unif} we get 
\begin{equation*}
|(\oa(t_{\alpha})-\omega(t),\phi)\lesssim \|\phi\|_{s,2}(\|\oa_0-\omega_0\|_2+\sqrt{\alpha}+|t_{\alpha}-t|). 
\end{equation*}
Letting $\alpha\to 0$ we deduce that 
\begin{equation}\label{eq:maintaweak}
\oa(t_{\alpha})\weakto\omega(t)\mbox{ in }L^{2}(\T^{2}). 
\end{equation}
Then \eqref{eq:mainta} follows from \eqref{eq;enstrophyth} and \eqref{eq:maintaweak} by using the very same argument used to prove \eqref{eq:stronpoint}. Indeed, 
\begin{equation*}
\begin{aligned}
\|\omega(t)\|_2&\leq \liminf_{\alpha\to 0}\|\oa(t_{\alpha})\|_2\\
&\leq \lim_{\alpha\to 0}\left(\|\oa_0\|_2^2+\alpha\|\nabla\oa_0\|_2^2\right)^{\frac{1}{2}}\\
&=\|\omega_0\|_2=\|\omega(t)\|_2.
\end{aligned}
\end{equation*}
\end{proof}
\section{Proof of Theorem \ref{teo:yudovich}}\label{sec:4}
In this section we prove Theorem \ref{teo:yudovich}. We start by recalling the two main tools we use in the proof. We first recall the following classical refinement of Gr\"onwall lemma. 
\begin{lemma}[Osgood Lemma]\label{lemma:osgood}
	Let $\rho$ be a positive Borel function, $\gamma$ a locally integrable positive function, $\mu$ a continuous positive increasing function and $\mathcal{M}(x) := \int_x^1\frac{dr}{\mu(r)}$. Let us assume that, for a strictly positive number $\eta$, the function $\rho$ satisfies
	\begin{equation*}
		\rho(t) \leq \eta+\int_{t_0}^t \gamma(\tau)\mu(\rho(\tau))\,d\tau.
	\end{equation*}
	Then, we have
	\begin{equation*}
	\mathcal{M}(\eta)-\mathcal{M}(\rho(t))\leq\int_{t_0}^t\gamma(\tau)\,d\tau.
\end{equation*}
\end{lemma}
\noindent For a proof of the lemma we refer to \cite{C95}. The following Gagliardo-Nirenberg inequality will be employed: there exists a constant $C>0$ such that, for any $f\in H^{1}(\T^2)$ with zero average and any $p>2$, it holds that 
\begin{equation}\label{eq:gs}
\|f\|_{\frac{2p}{p-1}}\leq C\|f\|_{2}^{1-\frac{1}{p}}\|\nabla f\|_{2}^{\frac{1}{p}}, 
\end{equation}
we refer to \cite{KW08} for the proof. Note that the constant $C$ in \eqref{eq:gs} is independent on $p$. 
\begin{proof}[Proof of Theorem \ref{teo:yudovich}]
Taking the difference of \eqref{eq:ev} and \eqref{eq:e} we get that 
\begin{equation*}
\partial_t(\ua-u)-\alpha\Delta\partial_t\ua+\ua\cdot\nabla(\ua-u)+(\ua-u)\cdot\nabla u=0.
\end{equation*}
Taking the $L^2$ scalar product with $\ua-u$, we have 
\begin{equation}\label{eq:scalarl2}
\frac{d}{dt}\|\ua-u\|_2^2+2\alpha(\nabla\partial_t\ua,\nabla(\ua-u))=-2\int(\ua-u)\cdot \nabla u(\ua-u)\,dx. 
\end{equation}
Note that \eqref{eq:scalarl2} is rigorously justified in the regularity class of Yudovich weak solutions. 
We need to manipulate the second term from the left-hand side. By using \eqref{eq:enstrophy} we have that 
\begin{equation*}
\begin{aligned}
2\alpha(\nabla\partial_t\ua,\nabla(\ua-u))&=2\alpha(\partial_t\nabla(\ua-u),\nabla(\ua-u))+2\alpha(\partial_t\nabla u,\nabla\ua)\\
&=\alpha\frac{d}{dt}\|\nabla\ua-\nabla u\|_{2}^2-2\alpha(\partial_tu,\Delta\ua), 
\end{aligned}
\end{equation*}
and then by using \eqref{eq:e} and the divergence-free condition we obtain 
\begin{equation*}
\frac{d}{dt}\left(\|\ua-u\|_2^2+\alpha\|\nabla\ua-\nabla u\|_2^2\right)=-2\int(\ua-u)\nabla u(\ua-u)\,dx-2\alpha\int\,u\cdot\nabla\,u\Delta\ua\,dx.
\end{equation*}
By integrating in time and ignoring the positive term $\alpha\|\nabla\ua(t)-\nabla u(t)\|_2^2$, we infer 
\begin{equation}\label{eq:mainestimate}
\begin{aligned}
&\|\ua(t)-u(t)\|_2^2\lesssim \|\ua_0-u_0\|_2^2+\alpha\|\nabla\ua_0-\nabla u_0\|_2^2\\
&+\int_0^t\int|\ua-u||\nabla u||\ua-u|\,dxd\tau+2\alpha\int_0^t\int|u||\nabla\,u||\Delta\ua|\,dxd\tau. 
\end{aligned}
\end{equation}
Next, we estimate the two integrals on the right-hand side of \eqref{eq:mainestimate}. For the first one we have that 
\begin{equation*}
\begin{aligned}
\int_0^t\int|\ua-u||\nabla u||\ua-u|\,dxd\tau&\lesssim \int_0^t\|\ua-u\|_{\frac{2p}{p-1}}\|\nabla u\|_{2p}\|\ua-u\|_2\,d\tau\\
&\lesssim \int_0^t\|\ua-u\|_2^{1-\frac{1}{p}}\|\nabla\ua-\nabla u\|_{2}^{\frac{1}{p}}\|\nabla\,u\|_{2p}\|\ua-u\|_2\,d\tau\\
&\lesssim 2p\int_0^t\|\ua-u\|_2^{2-\frac{1}{p}}\,d\tau,
\end{aligned}
\end{equation*}
where we have used H\"older inequality, the Gagliardo-Nirenberg inequality \eqref{eq:gs}, and \eqref{eq:cz}. Concerning the second term, by using H\"older inequality, Sobolev embedding and \eqref{eq:cz}, and Proposition \ref{prop:unif} we have that 
\begin{equation*}
\begin{aligned}
2\alpha\int_0^t\int|u||\nabla\,u||\Delta\ua|\,dxd\tau
&\lesssim T\,\sqrt{\alpha}\sup_{0 \leq t \leq T}\left(\|u\|_{\infty}\|\nabla\,u\|_2\sqrt{\alpha}\|\Delta\ua\|_2\right)\\
&\lesssim T\,\sqrt{\alpha}\sup_{0 \leq t \leq T}\left(\|\omega\|_{\infty}\|\nabla\,u\|_2\sqrt{\alpha}\|\Delta\ua\|_2\right)\\
&\lesssim \sqrt{\alpha}. 
\end{aligned}
\end{equation*}
Finally, by using \eqref{eq:aid2} and taking $\alpha<1$ small enough, we obtain that 
\begin{equation*}
\|\ua(t)-u\|_2^2\lesssim 2p\int_0^t\|\ua(\tau)-u(\tau)\|_2^{2\left(1-\frac{1}{2p}\right)}\,d\tau+\sqrt{\alpha}+\|\ua_0-u_0\|_2^2. 
\end{equation*}
Defining 
\begin{equation*}
y_{\alpha}(t):=\|\ua(t)-u(t)\|_2^2\,\qquad \delta_{\alpha}:=\sqrt{\alpha}+\|\ua_0-u_0\|_2^2,
\end{equation*}
we have that,
\begin{equation}\label{eq:osgoodalpha}
y_{\alpha}(t)\lesssim 2p\int_0^t\,[y_{\alpha}(s)]^{\left(1-\frac{1}{2p}\right)}\,ds+\delta_{\alpha},
\end{equation}
and we can conclude as in \cite[Theorem 1.4]{C96}: Assuming that 
$y_{\alpha}(t)<1$, we choose $2p(\tau)=2-\ln(y_{\alpha}(\tau))$, and then from \eqref{eq:osgoodalpha} we get that 
\begin{align*}
  y_{\alpha}(t)&\leq C\delta_{\alpha}+\displaystyle\int_0^t
          C\big(2-\ln(y_{\alpha}(\tau)\big)\,y_{\alpha}(\tau)^{1-\frac{1}{2-\ln(y_{\alpha}(\tau))}}d
          \tau
  \\
        &\leq C\delta_{\alpha}+C\displaystyle\int_0^t(2-\ln(y_{\alpha}(\tau)))\,y_{\alpha}(\tau)\,d \tau,
\end{align*}
with $C=C(T,\|\omega_0\|_{\infty}, \sup_{\alpha}(\sqrt{\alpha}\|\nabla\oa_0\|_2))>0$. Then, by using Lemma~\ref{lemma:osgood} with
$$
\begin{aligned}
  \rho(t)&:=y_{\alpha}(t),\hspace{0.5cm}\alpha:=C\delta_{\alpha},\hspace{0.5cm}\gamma(t):=C,
\\
\mu(x)&:=x(2-\ln x),\hspace{0.5cm}\mathcal{M}(x):=\ln(2-\ln x)-\ln 2,
\end{aligned}
$$
we obtain that
\begin{equation*}
-\ln(2-\ln y_{\alpha}(t))+\ln(2-\ln\delta_{\alpha})\leq C\,t ,
\end{equation*}
which implies that
\begin{equation*}\label{eq:rate_y_nu}
y_{\alpha}(t)\leq \exp\left(2-2\mathrm{e}^{-c\,t}\right)\left(\delta_{\alpha}\right)^{\mathrm{e}^{-C\,t}}\leq C_1\left(\delta_{\alpha}\right)^{\mathrm{e}^{-C_2 T}}.
\end{equation*}
Thus, for some constants $C_1,C_2>0$ depending only on $T$, $\|\omega_0\|_{\infty}$, and $\sup_{\alpha}(\sqrt{\alpha}\|\nabla\oa_0\|_2)$ we have that 
\begin{equation*}
\sup_{0 \leq t \leq T}\|\ua(t)-u(t)\|_{2}^{2}\lesssim (\sqrt{\alpha}+\|\ua_0-u_0\|_{2}^2)^{\mathrm{e}^{-C_1\,T}},
\end{equation*}
provided $\sqrt{\alpha}+\|\ua_0-u_0\|_{2}^2\leq C_2$ and we conclude. 
\end{proof}
\section{Proof of Theorem \ref{teo:strong}}\label{sec:5}
In this section we prove Theorem \ref{teo:strong}. 
We define  
\begin{equation}\label{eq:defcs}
C_{u_0}^{s,T}:=C\|u_{0}\|_{s,2}\exp\left((1+2\log^{+}\|u_0\|_{s,2})\mathrm{e}^{CT\|u_{0}\|_{s,2}}-1\right).
\end{equation}
By Sobolev embedding and Theorem \ref{teo:grhse} we have that 
\begin{equation}\label{eq:ghs}
\sup_{0 \leq t \leq T}\|u(t)\|_{s,2}\leq \Cs. 
\end{equation}
\begin{proof}[Proof of Theorem \ref{teo:strong}]
We divide the proof in two steps.\\

\noindent{\em Step 1: Convergence of the velocity.}\\

Taking the difference between \eqref{eq:ev} and \eqref{eq:e}, arguing as in the proof of Theorem \ref{teo:yudovich}, we obtain that 
\begin{equation*}
\frac{d}{dt}\left(\|\ua-u\|_2^2+\alpha\|\nabla\ua-\nabla u\|_2^2\right)=-2\int(\ua-u)\nabla u(\ua-u)\,dx-2\alpha\int\,u\cdot\nabla\,u\Delta\ua\,dx.
\end{equation*}
By using Sobolev embedding and \eqref{eq:ghs} we have 
\begin{equation*}
\int|\ua-u|^2|\nabla u|\,dx\lesssim \|u(t)\|_{s,2}\|\ua(t)-u(t)\|_2^2\leq \Cs\|\ua(t)-u(t)\|_2^2.
\end{equation*}
Concerning the second term, by using that $(u\cdot\nabla u,\Delta u)=0$, we first note that
\begin{equation*}
2\alpha\int u\cdot\nabla u\Delta\ua\,dx=2\alpha\int u\cdot\nabla u\,\Delta(\ua-u)\,dx,
\end{equation*}
and, integrating by parts using Proposition \ref{prop:unif} and \eqref{eq:ghs}, we have 
\begin{equation*}
\begin{aligned}
2\alpha\left|\int u\cdot\nabla u\Delta(\ua-u)\,dx\right|&\lesssim \alpha\int|\nabla u|^2|\nabla\ua-\nabla u|\,dx\\
&+\alpha\int|u||D^2u||\nabla\ua-\nabla u|\,dx\\
&\lesssim T(\Cs)^2\,\alpha\lesssim\alpha.
\end{aligned}
\end{equation*}
Thus, 
\begin{equation*}
\frac{d}{dt}\|\ua-u\|_{2}^{2}\lesssim \|\ua-u\|_{2}^{2}+\alpha
\end{equation*}
and by Gr\"onwall lemma we have 
\begin{equation*}
\sup_{0 \leq t \leq T}\|\ua-u\|_2\lesssim [\alpha+\|\ua_0-u_0\|^2_2]^{\frac{1}{2}}.
\end{equation*}
\noindent{\em Step 2: Convergence of the vorticity.}\\

We first consider the case $u_0\in H^{s}(\T^2)$, with $s\geq 3$. By taking the difference between the vorticity equation of \eqref{eq:ev} and the one of \eqref{eq:e} we obtain 
\begin{equation*}
\partial_t(\oa-\omega)-\alpha\Delta\partial_t\oa+\ua\cdot\nabla\oa-u\cdot\nabla\omega=0. 
\end{equation*}
Then, by suitable manipulations, we have 
\begin{equation}\label{eq:difom}
\partial_t(\oa-\omega)-\alpha\Delta\partial_t(\oa-\omega)+
\ua\cdot\nabla(\oa-\omega)=-(\ua-u)\cdot\nabla\omega
+\alpha\Delta(u\cdot\nabla \omega).
\end{equation}
Taking the $L^2$-scalar product of \eqref{eq:difom} with $\oa-\omega$ we have 
\begin{equation*}
\begin{aligned}
\frac{d}{dt}\left(\|\oa-\omega\|_2^2+\alpha\|\nabla\oa-\nabla\omega\|_2^2\right)&=-2\int(\ua-u)\cdot\nabla\omega\,(\oa-\omega)\,dx\\
&-2\alpha\int\nabla(u\cdot\nabla \omega)(\nabla\oa-\nabla\omega)\,dx.
\end{aligned}
\end{equation*}
Thus, by integrating in time and using \eqref{eq:aid2} we obtain 
\begin{equation}\label{eq:mainvorticity}
\begin{aligned}
\|\oa-\omega\|_2^2+\alpha\|\nabla\oa-\nabla\omega\|_2^2&\lesssim \|\oa_0-\omega_0\|_2^2+\alpha\|\nabla\oa_0-\nabla\omega_0\|_2^2\\
&+\int_0^t\int|\ua-u||\nabla\omega||\oa-\omega|\,dxd\tau\\
&+\alpha\int_0^t\int|\nabla u||\nabla\omega||\nabla\oa-\nabla\omega|\,dxd\tau\\
&+\alpha\int_0^t\int|u||D^2\omega||\nabla\oa-\nabla\omega|\,dxd\tau. 
\end{aligned}
\end{equation}
By using H\"older inequality and \eqref{eq:lad} we have that 
\begin{equation*}
\begin{aligned}
&\int_0^t\int|\ua-u||\nabla\omega||\oa-\omega|\,dxds
\lesssim\int_0^t\|\ua-u\|_4\|\nabla\omega\|_4\|\oa-\omega\|_2\,d\tau\\
&\lesssim \int_0^t\|\ua-u\|_2^{\frac{1}{2}}\|\ua-u\|_{1,2}^{\frac{1}{2}}
\|u\|_{2,2}^{\frac{1}{2}}\|u\|_{3,2}^{\frac{1}{2}}\|\oa-\omega\|_{2}\,d\tau\\
&\lesssim\int_0^t\|u\|_{s,2}\|\oa-\omega\|_{2}^{2}\,ds\lesssim\Cs\int_0^t\|\oa-\omega\|_{2}^{2}\,d\tau\lesssim \int_0^t\|\oa-\omega\|_{2}^{2}\,d\tau.
\end{aligned}
\end{equation*}
Moreover, by Proposition \ref{prop:unif}, \eqref{eq:aid2}, Sobolev embedding, and Young's inequality we have that 
\begin{equation*}
\begin{aligned}
\alpha\int_0^t\int|\nabla
u||\nabla\omega||\nabla\oa-\nabla\omega|\,dxd\tau
\lesssim\,&\alpha\int_0^t\|\nabla
u\|_{\infty}\|\nabla\omega\|_{2}\|\nabla\oa-\nabla\omega\|_2\,d\tau
\\
\lesssim\,&
T(\Cs)^4\alpha+\alpha\int_0^t\|\nabla\oa-\nabla\omega\|_2^2\,d\tau
\\
\lesssim\,& \alpha+\alpha\int_0^t\|\nabla\oa-\nabla\omega\|_2^2\,d\tau,
\end{aligned}
\end{equation*}
and 
\begin{equation*}
\begin{aligned}
\alpha\int_0^t\int|u||D^2\omega||\nabla\oa-\nabla\omega|\,dxd\tau&\lesssim \alpha\int_0^t\|u\|_{\infty}\|u\|_{s,2}\|\nabla\oa-\nabla\omega\|_2\,d\tau\\
&\lesssim T(\Cs)^4\alpha+\alpha\int_0^t\|\nabla\oa-\nabla\omega\|_2^2\,d\tau\\
&\lesssim \alpha+\alpha\int_0^t\|\nabla\oa-\nabla\omega\|_2^2\,d\tau.
\end{aligned}
\end{equation*}
Collecting the previous estimate we obtain 
\begin{equation*}
\begin{aligned}
\|\oa-\omega\|_2^2+\alpha\|\nabla\oa-\nabla\omega\|_2^2&\lesssim \int_0^t\left(\|\oa-\omega\|_2^2+\alpha\|\nabla\oa-\nabla\omega\|_2^2\right)\,d\tau\\
&+\alpha+\|\oa_0-\omega_0\|_2^2+\alpha\|\nabla\oa_0-\nabla\omega_0\|_2^2,
\end{aligned}
\end{equation*}
and by Gr\"onwall lemma 
\begin{equation*}
\|\oa-\omega\|_2^2\lesssim (\alpha+\|\oa_0-\omega_0\|_2^2+\alpha\|\nabla\oa_0-\nabla\omega_0\|_2^2)^{\frac{1}{2}}.
\end{equation*}
Next, we consider the case $2<s<3$. The argument is slightly more involved. Given $u_0\in H^{s}(\T^2)$ with $2<s<3$, we define $\un_0$ as 
\begin{equation*}
\un_0(x)=\sum_{0<|k|\leq N}\widehat{u_0}_{k}\,\mathrm{e}^{
    ik\cdot
    x},
\end{equation*}
where $\widehat{u_0}_{k}$ is the Fourier coefficients of $u_0$
associated with the wave-number $k$. Given $s\in(2,3)$ it holds that
\begin{equation}\label{eq:nest}
\begin{aligned}
&\|\un_0\|_{s,2}\leq \|u_0\|_{s,2},\\
&\|\un_0\|_{s',2}\leq N^{s'-s}\|u_0\|_{s,2},\mbox{ $s'>s$},\\
&\|\un_0-u_0\|_{\bar{s},2}\leq \frac{\|u_0\|_{s,2}}{N^{s-\bar{s}}},\mbox{ $0\leq\bar{s}<s$}.
\end{aligned}
\end{equation}
Let $\un$ be the unique classical solution of \eqref{eq:e} with initial datum $\un_0$ given by Theorem \ref{teo:grhse}, namely 
\begin{equation*}
\begin{cases}
\partial_t\un+\un\cdot\nabla\un+\nabla p^N=0\\
\diver\un=0\\
\un|_{t=0}=\un_0.
\end{cases}
\end{equation*}
We first note that by using \eqref{eq:dexp}, \eqref{eq:nest}, and
\eqref{eq:defcs} we have that the unique solution $u^N$ satisfies 
\begin{equation}\label{eq:sn}
\sup_{0 \leq t \leq T}\|u^N(t)\|_{s,2}\lesssim \Csn\lesssim \Cs.
\end{equation}
Moreover, a standard higher-order energy estimate gives
\begin{equation}\label{eq:3n}
\partial_t\|\un(t)\|_{3,2}\lesssim \|\nabla\un(t)\|_{\infty}\|\un(t)\|_{3,2}.
\end{equation}
By using Sobolev embedding, \eqref{eq:sn}, Gr\"onwall lemma, and \eqref{eq:nest} we have 
\begin{equation*}
\begin{aligned}
\sup_{0 \leq t \leq T}\|\un(t)\|_{3,2}&\lesssim \mathrm{e}^{T\Cs}\|\un_0\|_{3,2}\lesssim \mathrm{e}^{T\Cs}N^{3-s}\|u_0\|_{s,2}. 
\end{aligned}
\end{equation*}
Thus, 
\begin{equation*}\label{eq:main3n}
\sup_{0 \leq t \leq T}\|\un(t)\|_{3,2}\lesssim N^{3-s}.
\end{equation*}
Next, the equation for $\omn-\omega$ is 
\begin{equation}\label{eq:diffomn}
\partial_t(\omn-\omega)+\un\cdot\nabla(\omn-\omega)=-(\un-u)\cdot\nabla\omega.
\end{equation}
Taking the $L^2$-scalar product of \eqref{eq:diffomn} with $\omn-\omega$ and integrating in time we obtain 
\begin{equation*}
\|\omn-\omega\|_2^2\lesssim \int_0^t\int|\un-u||\nabla\omega||\omn-\omega|\,dxd\tau+\|\omn_0-\omega_0\|_2^2. 
\end{equation*}
By using H\"older inequality, \eqref{eq:gs} and Sobolev embedding we have 
\begin{equation}\label{eq:mainn}
\begin{aligned}
&\int_0^t\int|\un-u||\nabla\omega||\omn-\omega|\,dxd\tau\\
\lesssim&\,\int_0^t\|\un-u\|_{\frac{2}{s-2}}\|\nabla\omega\|_{\frac{2}{3-s}}\|\omn-\omega\|_2\,d\tau\\
\lesssim&\, \int_0^t\|\un-u\|_2^{s-2}\|\omn-\omega\|_2^{3-s}\|u\|_{s,2}\|\omn-\omega\|_2\,d\tau\\
\lesssim&\, \int_0^t\|u\|_{s,2}\|\omn-\omega\|_2^2\,ds\lesssim\Cs\int_0^t\|\omn-\omega\|_2^2\,d\tau\\
\lesssim&\,\int_0^t\|\omn-\omega\|_2^2\,d\tau.
\end{aligned}
\end{equation}
Therefore 
\begin{equation*}
\|\omn-\omega\|_2^2\lesssim \int_0^t\|\omn-\omega\|_2^2\,d\tau+\|\omn_0-\omega_0\|_2^2,
\end{equation*}
and by using Gr\"onwall lemma and \eqref{eq:nest} we get 
\begin{equation}\label{eq:fin}
\sup_{0 \leq t \leq T}\|\omn-\omega\|_2^2\lesssim \|\omn_0-\omega_0\|_2^2\lesssim\frac{1}{N^{2(s-1)}}. 
\end{equation}
Next, we estimate the $L^{2}$-norm of $\omn-\oa$. Obviously the inequality \eqref{eq:mainvorticity} holds with $\omn$ replacing $\omega$, namely 
\begin{equation*}\label{eq:mainvorticityn}
\begin{aligned}
\|\oa-\omn\|_2^2+\alpha\|\nabla\oa-\nabla\omn\|_2^2&\lesssim \|\oa_0-\omega^N_0\|_2^2+\alpha\|\nabla\oa_0-\nabla\omega^N_0\|_2^2\\
&+\int_0^t\int|\ua-\un||\nabla\omn||\oa-\omn|\,dxd\tau\\
&+\alpha\int_0^t\int|\un||D^2\omn||\nabla\oa-\nabla\omn|\,dxd\tau\\
&+\alpha\int_0^t\int|\nabla\un||\nabla\omn||\nabla\oa-\nabla\omn|\,dxd\tau\\
&=I_1+I_2+I_3+I_4+I_5. 
\end{aligned}
\end{equation*}
 We estimate the terms $I_i$'s separately. Concerning $I_1$, by triangular inequality and \eqref{eq:nest} we have 
 \begin{equation*}
 \begin{aligned}
 I_1&\lesssim \|\omn_0-\omega_0\|_2^2+\|\oa_0-\omega_0\|_2^2\\
 &\lesssim \frac{1}{N^{2(s-1)}}+\|\oa_0-\omega_0\|_2^2.
 \end{aligned}
 \end{equation*}
 Similarly, for $I_2$ we have 
\begin{equation*}
 \begin{aligned}
 I_2&\lesssim \alpha\|\nabla\oa_0-\nabla\omega_0\|_2^2+\alpha\|\nabla\omega_0-\nabla\oa_0\|_2^2\\
&\lesssim \frac{\alpha}{N^{2(s-2)}}+\alpha\|\nabla\omega_0-\nabla\oa_0\|_2^2.
\end{aligned}
\end{equation*}
Regarding $I_3$, arguing exactly as in \eqref{eq:mainn} we obtain that 
\begin{equation*}
\begin{aligned}
I_3&\lesssim \int_0^t\|\un\|_{s,2}\|\omn-\oa\|_2^{2}\,d\tau.
\end{aligned}
\end{equation*}
Then, by using \eqref{eq:sn} we have  that 
\begin{equation*}
I_3\lesssim \int_0^t\|\omn-\oa\|_2^{2}\,ds.
\end{equation*}
Next, we consider $I_4$. By using H\"older inequality, Sobolev embeddings and Young's inequality we have 
\begin{equation*}
\begin{aligned}
I_4&\lesssim \alpha\int_0^t\|\un\|_{\infty}\|D^2\omn\|_2\|\nabla\omn-\nabla\oa\|_2\,d\tau\\
&\lesssim \alpha\,T\sup_{0 \leq t \leq T}(\|\un\|_{s,2}^2\|\un\|_{3,2}^2)+\alpha\int_0^t\|\nabla\omn-\nabla\oa\|_2^2\,d\tau. 
\end{aligned}
\end{equation*}
By using \eqref{eq:sn} and \eqref{eq:3n} we have 
\begin{equation*}
I_4\lesssim \alpha\,N^{2(3-s)}+\alpha\int_0^t\|\nabla\omn-\nabla\oa\|_2^2\,d\tau.
\end{equation*}
Finally, concerning $I_5$ by using H\"older inequality, Sobolev embeddings and Young's inequality we have 
\begin{equation*}
\begin{aligned}
I_5&\lesssim \alpha\int_0^t\|\nabla\un\|_{\infty}\|\nabla\omn\|_{2}\|\nabla\omn-\nabla\oa\|_2\,d\tau\\
&\lesssim \alpha\,T(\sup_{0 \leq t \leq T}\|\un\|_{s,2}^4)+\alpha\int_0^t\|\nabla\omn-\nabla\oa\|_2^2\,d\tau\\
&\lesssim \alpha+\alpha\int_0^t\|\nabla\omn-\nabla\oa\|_2^2\,d\tau.
\end{aligned}
\end{equation*}
Collecting the previous estimate we get 
\begin{equation*}
\begin{aligned}
\|\oa-\omn\|_2^2+\alpha\|\nabla\oa-\nabla\omn\|_2^2&\lesssim \|\oa_0-\omega_0\|_2^2+\alpha\|\nabla\omega_0-\nabla\oa_0\|_2^2\\
&+\alpha+\alpha\,N^{2(3-s)}+\frac{1}{N^{2(s-1)}}+\frac{\alpha}{N^{2(s-2)}}\\
&+ \int_0^t\|\omn-\oa\|_2^{2}+\alpha\|\nabla\omn-\nabla\oa\|_2^2\,ds\\
&\lesssim \|\oa_0-\omega_0\|_2^2\!+\!\alpha\|\nabla\omega_0-\nabla\oa_0\|_2^2\!\\
&+ \int_0^t\|\omn-\oa\|_2^{2}+\alpha\|\nabla\omn-\nabla\oa\|_2^2\,ds\\
&+\!\frac{1}{N^{2(s-1)}}\!+\!\alpha\,N^{2(3-s)}.
\end{aligned}
\end{equation*}
By Gr\"onwall lemma we obtain 
\begin{equation}\label{eq:finalan}
\|\oa-\omn\|_2^2\lesssim \left(\frac{1}{N^{2(s-1)}}+\alpha\,N^{2(3-s)}+\|\oa_0-\omega_0\|_2^2+\alpha\|\nabla\omega_0-\nabla\oa_0\|_2^2\right). 
\end{equation}
Therefore, from \eqref{eq:finalan} and \eqref{eq:fin}, we obtain 
\begin{equation*}
\sup_{0 \leq t \leq T}\|\oa-\omega\|_2^2\lesssim\left(\frac{1}{N^{2(s-1)}}+\alpha\,N^{2(3-s)}+\|\oa_0-\omega_0\|_2^2+\alpha\|\nabla\omega_0-\nabla\oa_0\|_2^2\right).
\end{equation*}
Choosing $N\sim \alpha^{-\frac{1}{4}}$ we obtain \eqref{eq:convuths2}.
\end{proof}
\subsection*{Acknowledgments} 
S. Abbate acknowledges partial support by INdAM-GNAMPA and the hospitality of the University of Basel, where part of this work was carried out. G. Crippa is supported by the Swiss National Science Foundation through the project 212573 FLUTURA (Fluids, Turbulence, Advection) and by the SPP 2410 “Hyperbolic Balance Laws in Fluid Mechanics: Complexity, Scales, Randomness (CoScaRa)” funded by the Deutsche Forschungsgemeinschaft (DFG, German Research Foundation) through the project 200021E 217527 funded by the Swiss National Science Foundation. L.C. Berselli is partially supported by INdAM-GNAMPA and by the project PRIN 2020 “Nonlinear evolution PDEs, fluid dynamics and transport equations: theoretical foundations and applications” 20204NT8W4\_\,003. S. Spirito acknowledges partial support by INdAM-GNAMPA, by the projects PRIN 2020 “Nonlinear evolution PDEs, fluid dynamics and transport equations: theoretical foundations and applications” 20204NT8W4\_\,003, PRIN2022 “Classical equations of compressible fluids mechanics: existence and properties of non-classical solutions” 20204NT8W4\_\,003, and PRIN2022-PNRR “Some mathematical approaches to climate change and its impacts” P20225SP98\_\,002. 

\subsection*{Author Contributions} All authors contributed equally to this paper.

\subsection*{Data Availability} No datasets were generated or analyzed during the current study.

\section*{Declarations}

\subsection*{Conflict of interest} The authors declare no competing interests.


\begin{thebibliography}{10}

\bibitem{ACS24}S. Abbate, G. Crippa, \& S. Spirito, Strong convergence of the vorticity and conservation of the energy for the $\alpha$-Euler equations, {\em Nonlinearity}, \textbf{37} (2024) 035012 (25pp).

\bibitem{ABCGJK24} D. Albritton, E. Bru\'e, M. Colombo, C. De Lellis, V. Giri, M. Janisch, \& H. Kwon. Instability and non-uniqueness for the 2D Euler equations, after M. Vishik, {\em Ann. Math. Stud.}, volume 219, Princeton University Press, Princeton, NJ, (2024).

\bibitem{BFR80}J. Bardina, J. F. Ferziger \& W. C. Reynolds, Improved subgrid scale models for large eddy simulation. {\em American Institute Of Aeronautics And Astronautics Paper}. \textbf{80} pp. 13-57 (1980)

\bibitem{BB12}L. C. Berselli \& L. Bisconti, On the structural stability of the Euler-Voigt and Navier-Stokes-Voigt models. {\em Nonlinear Anal.}. \textbf{75}, 117-130 (2012).

\bibitem{BIL06} L. C. Berselli, T. Iliescu, \& W. J. Layton, Mathematics of large Eddy simulation of turbulent flows, Scientific Computation, {\em Springer-Verlag}, Berlin, (2006).

  \bibitem{BKR2016}
L.~C. Berselli, T.-Y. Kim, and L.~G. Rebholz, {Analysis of a reduced-order
  approximate deconvolution model and its interpretation as a
  {N}avier-{S}tokes-{V}oigt regularization}, {\em Discrete Contin. Dyn. Syst. Ser. B}
\textbf{21}, 1027--1050 (2016)

\bibitem{BS24}L. C. Berselli \& S. Spirito, On the Fourier-Galerkin approximation of the solutions of the 2D Euler equations with bounded vorticity. {\em J. Hyperbolic Differ. Equ.}, {\bf 21} No. 3, 503--522, (2024).

\bibitem{BCK24} E. Bru\'e, M. Colombo, \& A. Kumar,  Flexibility of Two-Dimensional Euler Flows with Integrable Vorticity, {\em arXiv:2408.07934}, (2024).
\bibitem{BI17} A. V. Busuioc \& D. Iftimie, Weak solutions for the $\alpha$-Euler equations and convergence to Euler, {\em Nonlinearity}, {\bf30}, no. 12, 4534--4557 (2017).

\bibitem{BILL20} A. V. Busuioc, D. Iftimie, M. D. Lopes Filho, \& H. J. Nussenzveig Lopes, The limit $\alpha\to 0$ of the $\alpha$-Euler equations in the half-plane with no-slip boundary conditions and vortex sheet initial data, {\em SIAM J. Math. Anal.}, {\bf52}, no. 5, 5257--5286 (2020).

\bibitem{CLT06}Y. Cao, E. Lunasin, \& E. Titi, Global well-posedness of the three-dimensional viscous and inviscid simplified Bardina turbulence models. {\em Commun. Math. Sci.}. \textbf{4}, 823-848 (2006).

\bibitem{C96}J. Chemin, A remark on the inviscid limit for two-dimensional incompressible fluids. {\em Commun. Partial Differ. Equ.}. \textbf{21}, 1771-1779 (1996).

\bibitem{C95}J. Chemin, Fluides parfaits incompressibles. {\em Astérisque}., 177 (1995)

\bibitem{CLLS16}A. Cheskidov, M. C. Lopes Filho, H. J. Nussenzveig Lopes, \& R. Shvydkoy, Energy conservation in Two-dimensional incompressible ideal fluids, {\em Comm. Math. Phys.}, $\mathbf{348}$, 129-143 (2016).

\bibitem{C22}G. Ciampa, Energy conservation for 2d euler with vorticity in $L(log L)^{\alpha}$, {\em Comm. Math. Sci.}, {\bf20}, no. 3, 855--875 (2022).

\bibitem{CCS20} G. Ciampa, G. Crippa, \& S. Spirito, Weak solutions obtained by the vortex method for the
2D Euler equations are Lagrangian and conserve the energy, {\em J. Nonlinear Sci.}, {\bf30} (6), 2787--2820, (2020).

\bibitem{CCS21}G. Ciampa, G. Crippa, \& S. Spirito, Strong convergence of the vorticity for the 2D Euler equations in the inviscid limit. {\em Arch. Ration. Mech. Anal.}. \textbf{240}, 295-326 (2021).

\bibitem{CDE22}P. Constantin, T. D. Drivas, \& T. M. Elgindi, Inviscid limit of vorticity distributions in Yudovich class, {\em Comm. Pure Appl. Math.}, $\mathbf{75}$, 60--82, (2022).

\bibitem{DRP24} L. De Rosa \& J. Park, No anomalous dissipation in two-dimensional incompressible fluids, {\em arXiv:2403.04668}, (2024).

\bibitem{D91} J.-M. Delort, Existence de nappes de tourbillon en dimension deux, {\em J. Am. Math. Soc.} {\bf4} no .3, 553--586 (1991).

\bibitem{D15} S. A. Denisov, Double exponential growth of the vorticity gradient for the two-dimensional
Euler equation, {\em Proc. Amer. Math. Soc.}, {\bf143}, no. 3, 1199–1210 (2015).

\bibitem{DPL89}R. J. DiPerna \& P.-L. Lions, Ordinary differential equations, transport theory and Sobolev spaces, {\em Invent. Math.}, $\mathbf{98}$, 511-547, (1989).

\bibitem{DPM87}R. J. DiPerna \& A. J. Majda, Concentrations in regularizations for 2-D incompressible flow, {\em Comm. Pure Appl. Math.}, $\mathbf{40}$, 301-345, (1987).

\bibitem{GOP04}J. Guermond, J. Oden, \& S. Prudhomme, Mathematical perspectives on large eddy simulation models for turbulent flows. {\em J. Math. Fluid Mech.}. \textbf{6}, 194-248 (2004).

\bibitem{H02}D.D. Holm, Variational principles for Lagrangian-averaged fluid dynamics. {\em J. Phys. A}, {\bf 35}, 679--688, (2002).

\bibitem{HMR98}D. D. Holm, J.E. Marsden \& T.S. Ratiu, The Euler-Poincar\'e equations and semidirect products with applications to continuum theories. {\em Adv. Math.}, {\bf137}, 1--81, (1998).

\bibitem{LL03}W. Layton \& R. Lewandowski, A simple and stable scale-similarity model for large eddy simulation: energy balance and existence of weak solutions. {\em Appl. Math. Lett.}. \textbf{16}, 1205-1209 (2003).

\bibitem{LL06}W. Layton \& R. Lewandowski, On a well-posed turbulence model. {\em Discrete Contin. Dyn. Syst. Ser. B}. \textbf{6}, 111-128 (2006).

\bibitem{LT10} J. S. Linshiz \& E. S. Titi, On the convergence rate of the Euler-$\alpha$, an Inviscid second-grade complex fluid, Model to the Euler equations, {\em J. Stat. Phys.}, {\bf 138}, no. 1-3, 305--332, (2010).

\bibitem{LML06}M. C. Lopes Filho, A. L. Mazzucato \& H. J. Nussenzveig Lopes, Weak solutions, renormalized solutions and enstrophy defects in 2D turbulence, {\em Arch. Ration. Mech. Anal.}, $\mathbf{179}$, no. 3, 353-387 (2006).

\bibitem{LLTZ15} M. C. Lopes Filho, H. J. Nussenzveig Lopes, E. S. Titi, \& A. Zang, Convergence of the 2D Euler-$\alpha$ to Euler equations in the Dirichlet case: indiﬀerence to boundary layers, {\em Phys. D}, {\bf 292/293}, 51--61, (2015).

\bibitem{LaT10}A. Larios \& E. Titi, On the higher-order global regularity of the inviscid Voigt-regularization of three-dimensional hydrodynamic models. {\em Discrete Contin. Dyn. Syst. Ser. B}. \textbf{14}, 603-627 (2010).

\bibitem{K86}T. Kato, Remarks on the Euler and Navier-Stokes equations in $\R^2$, {\em Proc. Symp. Pure Math.}, {\bf45}, 1--7 (1986).

\bibitem{KW08}H. Kozono \& H. Wadade, Remarks on Gagliardo-Nirenberg type inequality with critical Sobolev space and BMO, {\em Math. Z.}, {\bf 259}, 935--950 (2008).

\bibitem{MB02}A. J. Majda \& A. Bertozzi, Vorticity and incompressible flow, Cambridge University Press, Cambridge, 2002.

\bibitem{MP94}C. Marchioro \& M. Pulvirenti, Mathematical theory of incompressible nonviscous fluids, Springer-Verlag, New York, 1994.

\bibitem{VW93}I. Vecchi \& S. Wu, On $L^1$-vorticity for 2-D incompressible flow, {\em Manuscripta Mathematica}, {\bf78} 403--412, (1993).

\bibitem{V18} M. Vishik. Instability and non-uniqueness in the cauchy problem for the euler equa-
tions of an ideal incompressible fluid. Part I, {\em Arxiv:1805.09426} (2018).

\bibitem{Y63} V. I. Yudovi\v{c}, Non-stationary flows of an ideal incompressible fluid, {\em \v{Z}. Vy\v{c}isl. Mat. i Mat. Fiz.}, {\bf3}, 1032--1066, (1963).

\end{thebibliography}
\end{document}